\documentclass[a4paper,12pt, reqno]{amsart}
\usepackage[utf8]{inputenc}
\usepackage{amssymb,amsmath}
\usepackage{amscd}
\usepackage[all]{xy}
\usepackage[onehalfspacing]{setspace}
\usepackage{graphicx}
\usepackage{url}  
\usepackage[a4paper, left=2cm, right=2cm, top=3.5cm, bottom=3cm]{geometry}
\usepackage{tikz-cd}
\usepackage{array}
\usepackage{verbatim}

\def\A{\mathbb A}

\def\P{\mathbb P}

\def\Q{\mathbb Q}

\def\Z{\mathbb Z}
\def\AA{\mathcal A}
\def\BB{\mathcal B}
\def\CC{\mathcal C}
\def\DD{\mathcal D}
\def\EE{\mathcal E}
\def\FF{\mathcal F}

\def\HH{\mathcal H}

\def\OO{\mathcal O}
\def\PP{\mathcal P}
\def\RR{\mathcal R}
\def\TT{\mathcal T}

\def\Hom{{\operatorname{Hom}}}
\def\RHom{{\operatorname{RHom}}}

\def\Ext{{\operatorname{Ext}}}

\def\Spec{{\operatorname{Spec}}}

\def\Cl{{\operatorname{Cl}}}
\def\ord{{\operatorname{ord}}}
\def\Pic{{\operatorname{Pic}}}

\def\Bl{{\operatorname{Bl}}}

\def\Gr{{\operatorname{gr}}}

\def\Qcoh{{\operatorname{Qcoh}}}

\def\Sing{{\operatorname{Sing}}}
\def\End{{\operatorname{End}}}

\def\MCM{{\operatorname{MCM}}}

\def\rk{{\operatorname{rk}}}

\def\Mod{{\operatorname{Mod}}}
\def\mod{{\operatorname{mod}}}

\def\wt{\widetilde}
\def\wh{\widehat}
\def\ol{\overline}
\def\ul{\underline}

\def\Dsg{\DD^{\mathrm{sg}}}

\def\Db{\DD^b}

\def\Dperf{\DD^{\mathrm{perf}}}
\def\Perf{\DD^{\mathrm{perf}}}

\def\Ksg{{\mathrm K}^{sg}}
\def\KKsg{{\mathbb K}^{sg}}
\def\Kt{{\mathrm{K}}}
\def\Gt{{\mathrm{G}}}
\def\br{{\mathrm{br}}}
\def\Br{{\mathrm{Br}}}
\DeclareMathOperator{\injdim}{\mathsf{inj.dim}}

\def\on{\text{ on }}

\theoremstyle{plain}
\newtheorem{dummy}{dummy}[section]
\newtheorem{conjecture}[dummy]{Conjecture}
\newtheorem{theorem}[dummy]{Theorem}
\newtheorem{proposition}[dummy]{Proposition}
\newtheorem{lemma}[dummy]{Lemma}
\newtheorem{corollary}[dummy]{Corollary}
\newtheorem{example}[dummy]{Example}
\newtheorem{application}[dummy]{Application}
\newtheorem{definition}[dummy]{Definition}
\newtheorem{remark}[dummy]{Remark}
\numberwithin{equation}{section}

\def\bal{\begin{aligned}}
\def\eal{\end{aligned}}

\newcommand{\red}[1]{\leavevmode{\color{red}{#1}}}
\newcommand{\blue}[1]{\leavevmode{\color{blue}{#1}}}

\newcolumntype{P}[1]{>{\centering\arraybackslash}p{#1}}
\newcolumntype{M}[1]{>{\centering\arraybackslash}m{#1}}

\title[Obstructions to semiorthogonal decompositions for 3-folds I]
{Obstructions to semiorthogonal decompositions for singular threefolds I: $\Kt$-theory}
\author{Martin Kalck, Nebojsa Pavic and Evgeny Shinder}

\begin{document}

\begin{abstract}
We investigate necessary conditions for Gorenstein projective varieties
to admit semiorthogonal decompositions introduced by Kawamata, with main
emphasis on threefolds
with isolated compound $A_n$
singularities.
We introduce obstructions coming from Algebraic $\Kt$-theory and
translate them into the concept of maximal nonfactoriality.

Using these obstructions we 
show that many classes of nodal threefolds do not admit Kawamata
type semiorthogonal decompositions. These include
nodal hypersurfaces and double solids, with the exception of
a nodal quadric, and del Pezzo threefolds of degrees $1 \le d \le 4$ with maximal class group rank.

We also investigate when does a blow up of a smooth threefold
in a singular curve admit a Kawamata type semiorthogonal decomposition and we give a complete answer to this question when the curve is nodal and has only rational components.
\end{abstract}

\maketitle

\section{Introduction}

Semiorthogonal decompositions for derived categories of singular projective algebraic
varieties have recently began to be extensively studied.
One important type of such semiorthogonal decomposition is
\begin{equation}\label{eq:dec}
\Db(X) = \langle \Db(R_1), \dots, \Db(R_m) \rangle
\end{equation}
where $X/k$ is a projective variety and all $R_i$'s are
finite-dimensional $k$-algebras. 
One can think of \eqref{eq:dec} as a generalization of a full exceptional
collection which is the case when all $R_i = k$.

A typical construction of \eqref{eq:dec} proceeds through constructing
a full exceptional collection on a resolution of singularities
$\pi: \wt{X} \to X$ and pushing
it forward to $X$.
Burban has constructed decompositions \eqref{eq:dec}
for nodal chains of rational curves \cite{burban}, 
while Kawamata \cite{kawamata1},
Kuznetsov \cite{kuznetsov-sextics}
and Karmazyn-Kuznetsov-Shinder \cite{karmazyn-kuznetsov-shinder}
studied rational surfaces
with isolated rational singularities;
the exhaustive
answer for toric surfaces
is given in \cite{karmazyn-kuznetsov-shinder}.
Finally Kawamata \cite{kawamata1, kawamata2} 
has also studied two examples of Fano threefolds with a single ordinary double point
which admit decomposition \eqref{eq:dec}.
These examples are the nodal quadric threefold and a blow
up of $\P^3$
in two points followed
by contraction of the proper preimage of a line
passing through the two
points (this variety
can be also described
as a nodal linear section of
a Segre embedding $\P^2 \times \P^2 \subset \P^8$), see
Example \ref{example:Kawamata}.

In this paper we investigate necessary conditions for \eqref{eq:dec} to hold
on a Gorenstein projective variety $X$. In fact we allow 
more general decompositions
\begin{equation}\label{eq:dec2}
\Db(X) = \langle \AA, \Db(R_1), \dots, \Db(R_m) \rangle
\end{equation}
where $\AA \subset \Dperf(X)$ 
and which we call \emph{Kawamata type
semiorthogonal decompositions}
(because it is similar to what Kawamata has studied in \cite{kawamata2}).
Here again the $R_i$'s are finite-dimensional
$k$-algebras.
We assume that semiorthogonal decompositions
we consider are admissible; 
if $m=1$
the latter condition 
is automatic, see Proposition \ref{prop:admiss}.
We think of \eqref{eq:dec2} as a splitting of the derived
category into its ``nonsingular part" $\AA$ and the algebras $R_i$ which
carry
information about the singular
points of $X$.

We concentrate on obstructions coming from Algebraic $\Kt$-theory, 
namely on the negative $\Kt_{-1}(X)$ group.
The latter group is a part of the package of the
Thomason-Trobaugh K-theory machinery,
and the negative $\Kt$-groups including $\Kt_{-1}(X)$ have been extensively 
studied, in particular in the seminal work of Weibel \cite{weibel-surfaces}.

After recalling some preliminary results on
semiorthogonal decompositions and saturatedness in the singular setting,
Orlov's singularity category and various $\Kt$-theory groups
in Section \ref{sec:prelim}, 
in Section \ref{sec:K-theory} 
we translate vanishing of $\Kt_{-1}$ into geometric properties of $X$.
This has already been done by Weibel for curves and surfaces \cite{weibel-surfaces},
and our study concentrates on isolated threefold singularities, while
reproving some of Weibel's results
for curves and surfaces along the way.
This relies on previous joint work of the second
and third authors \cite{pavic-shinder}, where $\Kt$-theory of
Orlov's singularity category is studied.
We recall geometric description of
$\Kt_{-1}$ for curves in Proposition \ref{prop:K_-1-curve}
and Corollary \ref{cor:nodal}
and $\Kt_{-1}$ for surfaces
can be computed using Proposition \ref{prop:Class-groups-singularity}.

In general we show that 
vanishing of $\Kt_{-1}(X)$
implies that $X$  
\emph{has enough Weil divisors},
see Definition \ref{def:mnf} and Proposition \ref{prop:max-nonfact},
and that for certain types
of singularities,
including three-dimensional
compound $A_n$ singularities
vanishing of $\Kt_{-1}(X)$
is equivalent to $X$ 
having enough Weil divisors
(Corollary \ref{cor:K_{-1}}).

Informally, 
having enough Weil divisors
means 
having as many Weil non-Cartier
divisors as the local class groups allow.
In particular, in the nodal threefold
case
each local class group
is isomorphic to $\Z$, and
the condition of having enough Weil divisors
means that they separate
singularities, that is for every ordinary
double point $p \in X$
there exists a Weil divisor which
generates the local class group at $p$
and is Cartier at all other nodes.
This is stronger than just being non-factorial
which only requires existence of a Weil divisor which is non-Cartier.
We call nodal threefolds
which have enough Weil divisors
\emph{maximally nonfactorial}.


More generally we relate $\Kt_{-1}(X)$ to the so-called
\emph{defect} of $X$, that is the codimension
of $\Pic(X)$ in $\Cl(X)$, see Definition \ref{def:defect} and Corollary \ref{cor:K_{-1}}.
It follows that in the language of defect, maximal nonfactoriality for nodal
threefolds implies that
defect is equal to the number of singular points,
which is the maximal value the defect can take.

In Section \ref{sec:Kawamata} we show that existence of a decomposition \eqref{eq:dec2}
implies that $\Kt_{-1}(X) = 0$, see Corollary \ref{cor:K_{-1}-obstruction}.
This is obtained by passing to Orlov's
singularity category in \eqref{eq:dec2},
and using idempotent completeness
of the singularity category of a finite-dimensional
algebra.

Combining the results explained so far
we can state our main result as follows:
\begin{theorem}[Proposition \ref{prop:max-nonfact}
and
Corollary \ref{cor:K_{-1}-obstruction}]
If a normal 
Gorenstein projective variety $X$
has a Kawamata type decomposition
\eqref{eq:dec2},
then $\Kt_{-1}(X) = 0$.
If in addition $X$
has isolated singularities,
then $X$ has enough Weil divisors.
\end{theorem}

This explains why the two nodal threefolds
with a Kawamata type decomposition studied by Kawamata \cite{kawamata2} are nonfactorial. In both cases the threefold
$X$ has a single ordinary double point
with defect of $X$ being equal to one (in the nodal quadric threefold case $\Pic(X) = \Z$, $\Cl(X) = \Z^2$, 
while in the other example $\Pic(X) = \Z^2$, $\Cl(X) = \Z^3$), which illustrates the maximal
nonfactoriality of $X$.
Furthermore using the theorem above
we show that many types of threefolds do 
not admit decompositions \eqref{eq:dec2}.

\begin{application}[Example 
\ref{ex:hypers},
\ref{example:del-Pezzo-3folds-KSOD},
\ref{example:nodal-chain-blowup}]
The following types of nodal
threefolds do not admit
a Kawamata type semiorthogonal decomposition:
\begin{enumerate}
    \item All nodal threefold
hypersurfaces $X \subset \P^4$, except for the nodal quadric.

\item All nodal threefold
double solids $X \overset{2:1}{\to} \P^3$, 
except for the nodal
quadric.

\item Del Pezzo threefolds $V_d$ of degrees $1 \le d \le 4$ with maximal class group rank \cite{prokhorov}.

\item Threefolds obtained by
blowing up a nodal irreducible curve in a smooth threefold.
\end{enumerate}
\end{application}

Del Pezzo threefolds in
(3) can also be described
as follows \cite[Theorem 7.1]{prokhorov}: $V_d$ is a blow up
of $8-d$ general points on $\P^3$
followed by contraction of proper
preimages of lines
passing through pairs of points and twisted
cubics through six-tuples of 
points (for $d = 1, 2$).
Thus we negatively answer a question of Kawamata \cite[Remark 7.5]{kawamata2}, in all cases 
except for $d = 5$ which is a $3$-nodal $V_5$.
In fact we expect that only a few types of nodal Fano threefolds 
admit Kawamata type semiorthogonal decompositions. 
Looking at the potential cases of Fano threefolds
with maximal defect, I. Cheltsov has suggested the following.

\begin{conjecture}
The only nodal Fano threefolds of Picard rank one 
with Kawamata type decompositions
are the quadric, $V_5$ and $V_{22}$.
\end{conjecture}

However, in spite of the sparsity of the Fano examples,
we can construct lots of nodal threefolds
with a Kawamata decomposition using the blow up construction
with a locally complete
intersection center
as soon as the base variety
and the center of the blow up
both admit Kawamata type decompositions
(see Theorem \ref{thm:blowup}
and Corollary \ref{cor:blowup-ksod}).
In particular,
blowing up a smooth threefold in a disjoint
union of nodal trees
of smooth rational curves produces
nodal threefolds with
an arbitrary large
number of ordinary double
points and admitting a Kawamata type decomposition:

\begin{theorem}[Corollary {\ref{cor:nodal-blowup-characterisation}}]
Let $X$ be a smooth projective threefold
and $C$ is a disjoint union of nodal curves in $X$
such that all irreducible components of $C$ are rational curves. Then the blow up $\wt{X}$ of $X$ along $C$ admits a Kawamata type semiorthogonal decomposition if and only if $C$ is a disjoint union of nodal trees
with smooth rational components.
\end{theorem}

\subsection*{Relation to other work}
The link between idempotent completeness of the Orlov singularity category
and nonfactoriality is already present in the 
work of Iyama and Wemyss \cite{iyama-wemyss}.
It follows from \cite[Theorem 1.2]{iyama-wemyss}
that nodal threefolds with idempotent complete singularity categories
are nonfactorial. However from
the perspective of our applications
our results are sharper in a sense that we show
maximal nonfactoriality, which is strictly stronger than
nonfactoriality for varieties with several
ordinary double points.

The Grothendieck group of the singularity category has
been used by the first author of this paper and Karmazyn \cite[Corollary 5.3]{kalck-karmazyn}
to show that some types of surface quotient singularities
most notably $D_n$, $n \ge 4$ and $E_n$, $n = 6, 7, 8$ do not allow
a decomposition \eqref{eq:dec} with \emph{local} possibly noncommutative
algebras $R_i$'s. Even though all existing Kawamata type decompositions
for Gorenstein surfaces only admit $A_n$ singularities \cite{karmazyn-kuznetsov-shinder}, 
we do not currently know how to rule 
out $D_n$ and $E_n$ singularities without assuming that the algebras $R_i$ are local.

A similar
sort of obstruction to $\Kt_{-1}$
has been
used by Karmazyn, Kuznetsov
and the third
author of the present paper
\cite{karmazyn-kuznetsov-shinder},
where it is shown that a necessary condition
for existence of a decomposition
\eqref{eq:dec} on a projective normal rational
surface $X$ with rational singularities
is vanishing of the Brauer group $\Br(X)$.
We explain in Proposition \ref{prop:surface-Br}
that for such surfaces $\Br(X) \simeq \Kt_{-1}(X)$, so in this paper we
generalize
the obstruction from \cite{karmazyn-kuznetsov-shinder} 
from surfaces
to higher-dimensional varieties.

In the sequel to this paper \cite{KPS-new} we study restrictions on types of singularities
that are forced by Kawamata type decompositions, using 
representation theory of finite-dimensional algebras.

\subsection*{Acknowledgements} 
M.K. is deeply grateful to his family and to Wolfgang Soergel for the opportunity to work on this project.
He was partially supported by the GK1821 at the Universit\"at Freiburg; the latter grant
also allowed for a visit
of N.P. to Freiburg.
E.S. was partially supported by 
Laboratory of Mirror Symmetry NRU HSE, RF government grant, ag. N~14.641.31.0001.
N.P. and E.S. would like to thank the Max-Planck-Institut f\"ur Mathematik
in Bonn for the
excellent working conditions in which
much of this work has been planned
and discussed.
E.S. would also like to thank the
School of Mathematics and Statistics
at the University of Sheffield for providing
him with the Study Leave opportunity in
the second semester of 2019.

We thank I. Cheltsov,
J. Karmazyn,
Y. Kawamata,
A. Kuznetsov,
A. Kuznetsova,
D. Ploog, 
T. Raedschelders,
C. Shramov,
G. Stevenson,
J. Vitoria,
M. Wemyss
for useful conversations and e-mail correspondence
regarding this work,
and the referee for their comments
which significantly improved the write-up.

\section{Preliminaries}
\label{sec:prelim}

\subsection{Notation}
We work over an algebraically closed field $k$ of characteristic zero. 
By a scheme we mean a $k$-scheme satisfying
Orlov's ELF condition, that is a separated Noetherian
$k$-scheme
of finite Krull dimension
and having enough locally free coherent
sheaves.
By an (algebraic) variety we mean a reduced, but not necessarily irreducible, scheme of finite type over $k$. 

All triangulated categories are assumed to be k-linear. The opposite category of a category $\TT$ will be denoted $\TT^{\circ}$. We denote by $\DD(\Qcoh(X))$ the unbounded derived category of quasi-coherent sheaves, by $\Db(X)$ the bounded derived category of coherent sheaves of a variety $X$ and by $\Dperf(X)$ its full subcategory consisting of perfect complexes. Similarly, for a $k$-algebra $R$ we denote by $\Db(R)=\Db(\mod\text{-}R)$ the bounded derived category of finitely generated right modules over $R$ and $\Dperf(R)$ is again the full subcategory of perfect complexes in $\Db(R)$.

All functors such as pull-back $\pi^*$, pushforward $\pi_*$ and tensor product $\otimes$ when considered between derived categories are derived functors.

\subsection{Semiorthogonal decompositions and saturatedness}

Following \cite{bondal,bondal-kapranov, kuznetsov-hpd}, we recall standard definitions and properties of semiorthogonal decompositions of triangulated categories, of saturated categories and relations between these two notions.

Let $\TT$ be a triangulated category.
We call $\TT$ \emph{$\Hom$-finite} if $\dim_k\Hom(A,B)<\infty$ for all $A,B\in\TT$;
we call $\TT$ \emph{of finite type} 
if $\bigoplus_i\dim_k\Hom(A,B[i])<\infty$ for all $A,B\in\TT$.
For example, if $X$ is projective, then
$\Db(X)$ is $\Hom$-finite, and $\Dperf(X)$
is of finite type.

A triangulated category $\TT$ is called \emph{idempotent complete} (or Karoubian) 
if every idempotent $e \in \Hom(A,A)$ gives rise to a direct sum decomposition of $A$.
It is well-known that for every triangulated category $\TT$
has a triangulated idempotent completion $\TT \subset \ol{\TT}$ \cite{balmer-schlichting}.

Let $\AA\subset\TT$ be a full triangulated subcategory. The left and right orthogonals to $\AA$ are defined as
\begin{align*}
^{\perp}\AA =\lbrace T\in\TT \mid \forall A\in\AA , \ \Hom(T,A)=0 \rbrace ,\\
\AA^{\perp}=\lbrace T\in\TT \mid \forall A\in\AA , \ \Hom(A,T)=0 \rbrace .
\end{align*}

\begin{definition}[\cite{bondal-orlov}]
A collection $\AA_1,\ldots ,\AA_m$ of full triangulated subcategories of $\TT$ is called a semiorthogonal decomposition if for all $1\leq i < j\leq m$
\[
\AA_i\subset\AA_j^{\perp}
\]
and if the smallest triangulated subcategory of $\TT$ containing $\AA_1 ,\ldots ,\AA_m$ coincides with $\TT$. We use the notation 
\[
\TT=\langle\AA_1 , \ldots , \AA_m\rangle
\]
for a semiorthogonal decomposition of $\TT$ with components $\AA_1 , \ldots , \AA_m$.
\end{definition}

The next Lemma is well-known and follows immediately from the definitions:

\begin{lemma}\label{lem:semiorth-complete}
If $\TT$ admits a semiorthogonal decomposition into components $\AA_1, \dots, \AA_m$
then $\TT$ is idempotent complete if and only if all $\AA_i$'s are idempotent complete.
\end{lemma}

\begin{definition}[\cite{bondal,bondal-kapranov}]
A full triangulated subcategory $\AA$ of $\TT$ is called left (resp. right) admissible, if the inclusion functor $\AA\subset\TT$ has a left (resp. right) adjoint. If $\AA$ is both left and right admissible, then we call $\AA$ admissible in $\TT$.
\end{definition}

\begin{lemma}[{\cite[Proposition 1.5]{bondal-kapranov}}]
Let $\AA$ be a full triangulated subcategory of $\TT$, then $\AA$ is left (resp. right) admissible in $\TT$ if and only if there is a semiorthogonal decomposition $\TT=\langle\AA , {^{\perp}\AA}\rangle$ (resp. $\TT=\langle\AA^{\perp} , \AA\rangle$).
\end{lemma}

\begin{definition}
We call a semiorthogonal decomposition $\TT=\langle \AA_1 ,\ldots ,\AA_m\rangle$ admissible if every $\AA_i$ is admissible in $\TT$.
\end{definition}

Admissible decompositions are called strong in \cite{kuznetsov-hpd}.
Let us recall in what follows the relation between (left/right) admissible subcategories and representability of (co)homological functors of finite type. 
Note that in the following definition we do not assume that our triangulated category $\TT$ is of finite type (which is assumed in \cite{bondal-kapranov}).

\begin{definition}[\cite{bondal-kapranov}]
A $\Hom$-finite triangulated category $\TT$ is called left (resp. right) saturated if 
any exact functor $\TT\to\Db(k)$ (resp. $\TT^{\circ}\to\Db(k)$) is representable. If $\TT$ is both left and right saturated, then we call $\TT$ saturated.
\end{definition} 

\begin{theorem}[Rouquier, Neeman]
If $X$ is a projective variety, then $\Db(X)$ is saturated. 
\end{theorem}

\begin{proof}
It is a result proved by Neeman
that covariant functors $\Db(X) \to \Db(k)$
are represented by perfect complexes
\cite[Theorem 0.2]{neeman-original}
(see also
\cite[Theorem 7.1]{neeman})
\footnote{
The same result has been earlier announced by Rouquier
\cite[Corollary 7.51 (ii)]{rouquier}
but its proof has a mistake.}.
This immediately implies left saturatedness of $\Db(X)$.
Let us explain how this representability
result together
with Grothendeick-Verdier
duality also yields
right saturatedness of $\Db(X)$.

We need some notation first.
Let us denote by $D:\Db(k)\to\Db(k)^{\circ}$ the dualizing functor $D(\text{-})=\RHom_{\Db(k)}(\text{-},k)$. Denote further by $\omega_X^{\bullet}\in\Db(X)$ the dualizing complex of $X$ (see Subsection \ref{subsection:Gorenstein} below) and write $\PP^{\vee}=\RR\HH om(\PP,\OO_X)\in\Db(X)$ for a perfect complex $\PP\in\Db(X)$. 
Finally, for a functor
$F: \CC \to \DD$ we write
$F^\circ: \CC^{\circ} \to \DD^{\circ}$
for the same functor between
the opposite categories. This way,
both
$D D^\circ$ and
$D^\circ D$ are canonically isomorphic
to the identity functors
on respective categories.
Taking the opposite 
preserves compositions of functors.

Let $F:\Db(X)^{\circ}\to\Db(k)$ be an exact functor. Note that
by definition representability
of $F$ is tautologically the
same as representability of $F^\circ$. 
Consider $(DF)^\circ$ which is a covariant exact functor on $\Db(X)$ and thus as explained above $(DF)^\circ = D^\circ F^\circ$ is represented by a perfect complex $\PP\in\Db(X)$. 
We compute
\begin{align*}
    F^\circ \simeq DD^\circ F^\circ
    &\simeq \RHom_{\Db(k)}( \RHom_{\Db(X)}(\PP ,\text{-}),k)\\ 
    &\simeq \RHom_{\Db(k)}( Rp_* ((\text{-}) \otimes \PP^{\vee}) , k )\\
    &\simeq \RHom_{\Db(X)}(  (\text{-}) \otimes \PP^{\vee} , \omega_X^{\bullet} )\\
    &\simeq \RHom_{\Db(X)}( \text{-},\PP\otimes\omega_X^{\bullet}) ,
\end{align*}
where we used the fact that $\PP$ is a perfect complex in the second and fourth equality and Grothendieck-Verdier duality with respect to the projection $p : X \to \Spec(k)$ in the third equality. Hence $F$ is represented by $\PP\otimes\omega_X^{\bullet}\in\Db(X)$.
\end{proof}








\begin{lemma}[\cite{bondal-kapranov}]\label{lem:admissible-implies-saturated}
Let $\TT$ be a $\Hom$-finite and saturated triangulated category and let $\AA$ be a left (resp. right) admissible full triangulated subcategory of $\TT$. Then $\AA$ is saturated.
\end{lemma}

\begin{proof}
See e.g. \cite[Lemma 2.10]{kuznetsov-hpd}.
\end{proof}

\begin{corollary}\label{cor:saturated-subcat-Db(X)}
Let $X$ be a projective variety. Then any left (resp. right) admissible subcategory of $\Db(X)$ is saturated. 
\end{corollary}

Finite type saturated categories are universally admissible in the following sense.

\begin{proposition}[{\cite[Proposition 2.6]{bondal-kapranov}}]\label{lem:saturated-implies-admissible}
Let $\AA$ be a full triangulated subcategory of $\TT$, where $\TT$ is of finite type and let moreover $\AA$ be left (resp. right) saturated. Then $\AA$ is left (resp. right) admissible in $\TT$.
\end{proposition}

\begin{definition}[\cite{bondal-kapranov}]
Let $\TT$ be a $\Hom$-finite triangulated category.
Then an autoequivalence $S: \TT \to \TT$ is called a Serre functor
if there is a functorial equivalence
\[\Hom(A, B) \simeq \Hom(B, S(A))^\star\]
for $A, B \in \TT$. Here $(-)^\star$ stands for the
vector space duality.
\end{definition}

\begin{lemma}[{\cite[Proposition 3.7]{bondal-kapranov}}]\label{lem:Serre-adm}
If $\TT$ has a Serre functor then for every admissible decomposition
$\TT = \langle \AA_1, \dots, \AA_m \rangle$
each component $\AA_i$ has a Serre functor.
\end{lemma}

\subsection{Gorenstein varieties and algebras}\label{subsection:Gorenstein}

\begin{definition}
A two-sided noetherian ring $R$ satisfying $\injdim _{R}\!R < \infty$ and 
$\injdim R_{R} < \infty$ is called Gorenstein. 
A variety $X$ is called Gorenstein if all its local rings are Gorenstein.
\end{definition}

Gorenstein property is preserved under regular embeddings, projective bundles
and blow ups with locally complete intersection centers.

Let $X$ be a projective variety and 
let $\omega_X^{\bullet} \in \Db(X)$ 
denote the dualizing complex $p^{!}(k)$ of $X$, where $p:X\to\Spec(k)$ (for the definition of $p^{!}$, see \cite{residues-duality}, and
$\omega_X^\bullet$ is bounded by \cite[Subchapter V.2]{residues-duality}). 
It is well-known that 
a connected projective variety
$X$ is Gorenstein if and only if $\omega_X^{\bullet}$ is a shift of a line bundle \cite[Proposition V.9.3]{residues-duality}.
Let $X$ be Gorenstein and connected, and consider 
$S_X(-)= (-)\otimes\omega_X [\dim(X)]$,
the Serre functor on $\Dperf(X)$. 
By abuse of notation, we also write $S_X$ for the autoequivalence on $\Db(X)$ defined by the same formula. 
$S_X$ is not a Serre functor on $\Db(X)$, however
by the Grothendieck-Verdier duality there is a weaker statement: for all $E$ in $\Dperf(X)$ and all $F$ in $\Db(X)$ (resp. for all $E$ in $\Db(X)$ and all $F$ in $\Dperf(X)$), we have
\begin{equation}\label{eq:serre-autoequiv}
    \Hom_{\Db(X)}(E , F)\simeq\Hom_{\Db(X)}(F , S_X(E))^{\star} .
\end{equation}
This isomorphism typically fails when neither $E$ nor $F$ are perfect, for instance
it always fails for structure sheaves of singular points.

The homological meaning of the Gorenstein condition is the following result:

\begin{lemma}\label{lem:Gor-Serre}
A projective variety $X$ (resp. finite-dimensional algebra $R$) 
is Gorenstein if and only if $\Dperf(X)$ (resp. $\Dperf(R)$)
has a Serre functor.
\end{lemma}
\begin{proof}
For finite-dimensional algebras this
is a result of Chen \cite[Corollary 3.9]{chen}, which goes back to Happel \cite[Section 3.6]{happel}.

For varieties the ``only if" direction is clear. 
For the ``if" direction, let us denote by $S$ the Serre functor on 
$\Dperf(X)$ and let $\omega_X^{\bullet}$ be as above. By the definition of the Serre functor $S$ and 
by Grothendieck-Verdier duality, we have a functorial
isomorphism
\begin{equation}\label{eqn:Serre-Grothendieck-Verdier}
  \Hom(E,\omega_X^{\bullet})\simeq\Hom(E, S(\OO_X))  
\end{equation}
for $E\in\Dperf(X)$. In particular we obtain a canonical map $f:S(\OO_X)\to\omega_X^{\bullet}$ corresponding to the identity morphism of $\OO_X$. 
Let $C$ be the cone of $f$. 
By $\eqref{eqn:Serre-Grothendieck-Verdier}$ we see that 
$\Hom(E,C)=0$ for all $E\in \Dperf(X)$. 
By \cite[Theorem 3.1.1 (2)]{bondal-van-den-bergh} 
the unbounded derived
category of quasicoherent
sheaves on $X$
has a perfect generator,
in particular 
the right orthogonal
$\Dperf(X)^{\perp}$ 
is trivial 
(considered either in
the unbounded
derived categories
of quasicoherent sheaves,
or in $\Db(X)$), 
and thus $C=0$. In other words, $S(\OO_X)\simeq\omega_X^{\bullet}$. In particular the dualizing complex of $X$ is perfect. Since 
${\omega_X^\bullet}^\vee \otimes \omega_X^\bullet \simeq \RR\HH om (\omega_X^{\bullet},\omega_X^{\bullet})\simeq\OO_X$
by the definition of a dualizing complex,
it is easy to deduce that $\omega_X^{\bullet}$ is a shift of a line bundle on the connected components of $X$ (see e.g. \cite[Lemma V.3.3]{residues-duality}).
Equivalently, $X$ is Gorenstein.
\end{proof}

In the Gorenstein case one can mutate semiorthogonal decompositions as follows:

\begin{lemma}\label{lem:Gor-admiss}
Let $X$ be a Gorenstein projective variety. If $\Db(X) = \langle \AA, \BB \rangle$, and
either $\AA$ or $\BB$ is contained in $\Dperf(X)$, then
both $\AA$ and $\BB$ are admissible and there is a semiorthogonal
decomposition $\Db(X) = \langle \BB \otimes \omega_X, \AA \rangle$.
\end{lemma}
\begin{proof}
Let $\AA\subset\Dperf(X)$. 
By Corollary 
$\ref{cor:saturated-subcat-Db(X)}$,
$\AA$ is saturated,
so by 
Proposition $\ref{lem:saturated-implies-admissible}$
$\AA$ is admissible in $\Dperf(X)$, hence
by Lemma
\ref{lem:Serre-adm},
$\AA$ has an induced Serre functor $S_\AA$
(alternatively, instead of relying on Lemma
\ref{lem:Serre-adm}
we can deduce existence
of the Serre functor on $\AA$ using \cite[Corollary 3.5]{bondal-kapranov} immediately
from $\AA$ being saturated and of finite type).

We can define the right adjoint of $I:\AA\to\Db(X)$ using the following standard construction of \cite{bondal-kapranov}. Let $L:\Db(X)\to\AA$ be the left adjoint of $I$ and define the functor $R:\Db(X)\to \AA$ by the formula $R=S_{\AA}\circ L\circ S_X^{-1}$. It follows
from definitions and \eqref{eq:serre-autoequiv} that $R$ is right adjoint to $I$, and that there is a semiorthogonal decomposition 
\[
\Db(X)=\langle \BB\otimes\omega_X ,\AA\rangle .
\]
Since $\omega_X$ is a line bundle, we obtain also $\langle\BB\otimes\omega_X,\AA \rangle\simeq\langle \BB , \AA\otimes\omega_X^{\vee} \rangle $ and hence $\BB\hookrightarrow\Db(X)$ is admissible as well.

The case $\BB\subset\Dperf(X)$ can be proven similarly. 
\end{proof}

\subsection{Singularity categories}

We recall standard facts about singularity categories.
The basic references for these results are \cite{buchweitz,orlov-sing-1}.
Let $X$ be $k$-scheme satisfying Orlov's ELF condition \cite{orlov-sing-1};
in particular
we can take $X$
to be a quasi-projective variety, or
the $\Spec$ of
a completion
of a local ring
for a point in a variety.
For every closed
$Z \subset X$
the triangulated category of singularities of $X$ supported at $Z$
is 
the Verdier quotient
\[
\Dsg_Z(X)=\Db_Z(X) /\Dperf_Z(X).
\]
We write $\Dsg(X)$ for $\Dsg_X(X)$.
If $R$ is a ring, then we define its singularity category by the same formula
$\Dsg(R) = \Db(R) / \Dperf(R)$.

Let us denote by $\ol{\Dsg(X)}$ the idempotent completion of $\Dsg(X)$. 
As we will see in the next section, idempotent completeness of $\Dsg(X)$
is controlled by the first negative $\Kt$-theory group of $X$.

The following is an important property of the singularity category, called Kn\"orrer periodicity.

\begin{theorem}[{\cite[Theorem 2.1]{orlov-sing-1}}]\label{Thm:Knorrer} 
Let $X$ be regular and let $f:X\to\A^1$ be a non-zero morphism. 
Define $g=f+xy:X\times\A^2\to\A^1$.
Let $Z_f=f^{-1}(\{0\})$ and $Z_g=g^{-1}(\left\lbrace 0\right\rbrace )$.
Then we have a canonical equivalence
$$\Dsg(Z_f) \simeq \Dsg(Z_g).$$
\end{theorem} 

The following result goes back to Auslander.

\begin{proposition}\label{prop:CY}
If $X$ is $n$-dimensional Gorenstein variety 
with only isolated singularities,
then $\Dsg(X)$ is a Calabi-Yau-$(n-1)$ category,
that is $[n-1]$ is its Serre functor, that is $\Dsg(X)$ is $\Hom$-finite
and
for every two objects $E, F \in \Dsg(X)$
we have a functorial isomorphism
\[
\Hom(E, F) \simeq \Hom(F, E[n-1])^{\star}.
\]
\end{proposition}
\begin{proof}
Since $X$ is Gorenstein
with isolated singularities,
$\Dsg(X)$ is $\Hom$-finite
by \cite[Corollary 1.24]{orlov-sing-1}.
Let us prove the Calabi-Yau property.
We can assume that $X$ is affine.
Indeed, removing a hyperplane section disjoint from the singular
locus, we obtain an open affine subvariety $U \subset X$
containing the singular locus,
and by \cite[Proposition 1.14]{orlov-sing-1} $\Dsg(U) \simeq \Dsg(X)$.
In this case the desired statement is  
a result of Auslander \cite[Theorem 3.1]{auslander} combined with Buchweitz' interpretation of the singularity
category \cite[Theorem 4.4.1 (2)]{buchweitz}
(see \cite[Theorem 1.4]{iyama-wemyss2} for a more
general statement).
\end{proof}

\begin{example}\label{ex:CY0}
If $Q$ is an affine nodal $n$-dimensional quadric (that is an ordinary double
point of dimension $n$), then by Kn\"orrer periodicity
Theorem \ref{Thm:Knorrer}, $\Dsg(Q)$
only depends on the parity of $n$.
In particular, let us assume $n \equiv 1 \pmod{2}$, so that
\[
\Dsg(Q) \simeq \Dsg(A),
\]
where $A = k[x,y]/(xy)$.
By Proposition \ref{prop:CY},
$\Dsg(Q)$
is a Calabi-Yau-$0$ category, that is 
$\Hom(E,F) \simeq \Hom(F,E)^{\star}$.
In fact $\Dsg(Q)$ is equivalent to the
category of $\Z/2$-graded finite-dimensional
vector spaces, with the shift functor $[1]$
exchanging the graded pieces.
\end{example}

\begin{lemma}[Chen {\cite[Corollary 2.4]{chen-completion}}]\label{lem:algebras-complete}
For any finite dimensional $k$-algebra $R$, $\Dsg(R)$ is idempotent complete.
\end{lemma}

In the Gorenstein case, the Lemma above
also follows from the famous result of Buchweitz 
that
$\Dsg(R)$ 
of a Gorenstein ring $R$ is equivalent to the stable category of stable maximal Cohen-Macaulay $R$-modules (also called Gorenstein projectives) $\underline{\mathrm{MCM}}(R)$
\cite[Theorem 4.4.1]{buchweitz},
and the latter category is well-known to be
idempotent complete for finite-dimensional algebras
(see e.g. \cite[Lemma 2.68]{martin-thesis}).

\subsection{Grothendieck groups
and the topological filtration}
We assume that $X$ is an ELF $k$-scheme.
Let $Z \subset X$ be a closed
subscheme.

We define the following Grothendieck groups
of $X$ with supports on $Z$;
the first two are classical, and the last two
are defined and studied in \cite{pavic-shinder}.
We define
\[\bal
\Kt_0(X \on Z) &= \Kt_0(\Dperf_Z(X)) \\
\Gt_0(X \on Z) &= \Kt_0(\Db_Z(X)) \simeq \Gt_0(Z) \\
\Ksg_0(X\text{ on }Z) & = \Kt_0(\Dsg_Z(X)) \\
\KKsg_0(X\text{ on }Z) & = \Kt_0(\ol{\Dsg_Z(X)}) \\
\eal\]
where the the isomorphism in the second line is Quillen's devissage.
The last two groups are called
\emph{singularity Grothendieck groups}.
We write $\Ksg_0(X)$ (resp. $\KKsg_0(X)$) for these groups
when $Z = X$.
Essentially from definitions
(see \cite[Remark 1.13]{pavic-shinder}) 
we get a canonical exact sequence
\begin{equation}\label{eq:KGseq}
\Kt_0(X\text{ on }Z)\to\Gt_0(Z)\to\Ksg_0(X\text{ on }Z)\to 0.
\end{equation}

Let $\Kt_{-1}(X)$ (resp. $\Kt_{-1}(X\text{ on }Z)$)
be the $(-1)$-st $\Kt$-group of $X$ (resp.
of $X$ with supports in $Z$) 
\cite{tt}. We have the following well-known relation between the two singularity Grothendieck
groups defined above,
going back to Thomason \cite{thomason}, Schlichting \cite{schlichting2} and Orlov \cite{orlov-sing-2}.

\begin{lemma}[{\cite[Lemma 1.10, 1.11 and Remark 1.13]{pavic-shinder}}]\label{lem:SES_K_{-1}}
There is a canonical short exact sequence 
\begin{align}\label{SES}
0\to\Ksg_0(X \text{ on } Z)\to\KKsg_0(X \text{ on } Z)\to\Kt_{-1}(X \text{ on } Z)\to 0. 
\end{align}
Moreover, $\Dsg_Z(X)$ is idempotent complete if and only if $\Kt_{-1}(X \text{ on } Z)=0$.
\end{lemma}

We note that all categories
and Grothendieck groups with supports
on $Z$ used above only depend on
the set of points of $Z$ rather than
its scheme structure.

Assume that $X$ is pure, that is
$X$ has no embedded components and 
all irreducible components
of $X$ have the same dimension.
There is a so-called \emph{topological filtration} $F^i\Ksg_0(X)$ on
$\Ksg_0(X)$ induced by the topological filtration on $\Gt_0(X)=\Kt_0(\Db(X))$ (see \cite[Subchapter 1.3]{pavic-shinder}). Recall that $F^i\Ksg_0(X)$ is generated by elements $[\OO_T]$, where $T\subset X$ is a closed subscheme of codimension at least $i$. Let us denote the associated graded groups by $\Gr^i\Ksg_0(X)$. A topological filtration on $\Ksg_0(X\text{ on }Z)$ can be defined in the same way. 
We have the following useful properties of the associated graded groups 
of $\Ksg_0(X)$:

\begin{proposition}[{\cite{pavic-shinder}}]\label{prop:filtration-Ksg1}
\noindent 1) Let $X$ be a purely $n$-dimensional scheme with only isolated singularities
and let $Z \subset X$ be a closed subscheme.
Then there is an isomorphism 
\begin{equation}\label{eqn:relative-Ksg-completion}
\KKsg_0(X\text{ on }Z)\simeq \bigoplus_{p\in\Sing(X)\cap Z}\Ksg_0(\wh{\OO}_{X,p}).    
\end{equation}
Furthermore for all $i \ge 0$ we have
\begin{equation}\label{eq:Fi-relative-completion}
F^i \Ksg_0(X \on Z) \subset
\bigoplus_{p\in\Sing(X)\cap Z}
F^i \Ksg_0(\wh{\OO}_{X,p}),
\end{equation}
and in particular, $F^n\Ksg_0(X\text{ on }Z)=0$.

\noindent 2) Let $C$ be reduced, connected 
purely one-dimensional scheme
with $N$ irreducible components and let $Z\subset C$ be a reduced subscheme of $C$ of dimension one. Denote by $N_Z$ the number of irreducible components of $Z$. Then 
\begin{equation*}
\Ksg_0(C\text{ on }Z) = \Gr^0\Ksg_0(C\text{ on } Z) = \begin{cases}
\Z^{N_Z} & Z\subsetneq C \\
\Z^{N_C -1} & Z = C
\end{cases} ,
\end{equation*}
generated by the structure sheaves of the irreducible components of $Z$.

\noindent 3) If $X$ is normal and irreducible then
$\Gr^1 \Ksg_0(X) \simeq \Cl(X)/\Pic(X)$,
functorially with respect to flat pullbacks.

\noindent 4) The isomorphism $\Ksg_0(Z_f) \simeq \Ksg_0(Z_g)$ induced by
Theorem \ref{Thm:Knorrer} shifts the topological filtration by one, 
that is for all $i \ge 0$ we have natural isomorphisms 
$F^i \Ksg_0(Z_f) \simeq F^{i+1} \Ksg_0(Z_g)$ and $\Gr^i \Ksg_0(Z_f) \simeq \Gr^{i+1} \Ksg_0(Z_g)$. 
\end{proposition}

\begin{proof}
1) Let us denote by $S$ the singular locus of $X$. We have a well-defined, fully-faithful functor $\Dsg_{Z\cap S}(X)\to\Dsg_Z(X)$ \cite[Lemma 2.6]{orlov-sing-2} and its image is dense in $\Dsg_Z(X)$ by \cite[Proposition 2.7]{orlov-sing-2}. 
In particular, these two categories have the same idempotent completion. Moreover, since $Z\cap S$ is a finite set of closed points, we have 
\[
\ol{\Dsg_{Z}(X)}\simeq\ol{\Dsg_{Z\cap S}(X)}\simeq\bigoplus_{p\in Z\cap S}{\Dsg
(\wh{\OO}_{X,p})},
\]
where we used in the second equivalence that formally isomorphic singularities have the same idempotent complete singularity categories \cite[Theorem 2.10]{orlov-sing-2} and that singularity categories of local Henselian
rings are idempotent complete 
(indeed, complete local rings are Henselian;
Henselian rings have vanishing $\Kt_{-1}$
by \cite[Theorem 3.7] {drinfeld2}
and hence 
Lemma \ref{lem:SES_K_{-1}}
implies that each $\Dsg(\wh{\OO}_{X,p})$
is idempotent complete).
Passing to the Grothendieck group yields \eqref{eqn:relative-Ksg-completion}. 

By \eqref{SES}, we get \eqref{eq:Fi-relative-completion}
for $i = 0$, and then since
flat pullbacks of morphisms preserve the 
topological filtration \cite[Lemma 1.29 (1)]{pavic-shinder},
\eqref{eq:Fi-relative-completion}
follows for all $i \ge 0$.

By definition $F^n \Gt_0(\wh{{\OO}}_{X,p})$
is generated by the simple module $[k]$,
in particular we have a surjective restriction
homomorphism $F^n \Ksg_0(X) \to 
F^n \Ksg_0(\wh{{\OO}}_{X,p})$
(sending $[k(x)]$ to $[k]$)
and by \cite[Proposition 1.24 (4)]{pavic-shinder} we know 
that $F^n \Ksg_0(X) = 0$, so that $F^n \Ksg_0(\wh{{\OO}}_{X,p}) = 0$ and by \eqref{eq:Fi-relative-completion},
$F^n \Ksg_0(X\text{ on }Z)=0$.

2) By 1) we see that $F^1 \Ksg_0(C\text{ on }Z)=0$, or, equivalently, that $\Ksg_0(C\text{ on }Z)\simeq \Gr^0\Ksg_0(C\text{ on }Z)$.
The result now follows using \eqref{eq:KGseq},
since classes of perfect complexes on $X$ with one-dimensional support contained in $Z$ generate $\Z\cdot [\OO_C]$
(resp. have trivial image) in $\Gr^0 \Gt_0(C) = \Z^{N_C}$
for $Z = C$ (resp. $Z \subsetneq C$).

3) For the isomorphism see \cite[Proposition 1.24 (2)]{pavic-shinder}, which works under
the quasi-projective assumption, but
the proof remains unchanged in the ELF scheme case.
Functoriality follows easily by construction.

4) See \cite[Proposition 1.30]{pavic-shinder}; same comment as above regarding
quasi-projective and ELF applies here.
\end{proof}

\subsection{Local geometry of compound $A_n$ singularities}\label{subsec:cAn}

Recall that a threefold $X$ has a \emph{compound $A_n$} (abbreviated as $cA_n$)
\emph{singularity} 
at $p \in X$ if the complete local ring $\wh{\OO}_{X,p}$ is isomorphic to a hypersurface singularity given by the equation
\[
f=xy+z^{n+1} + w h(x,y,z,w),
\]
where $h$ is an arbitrary power series (see \cite[Definition 2.1]{reid}). It is well-known in the isolated singularity case that the equation $f$ can, after a change of coordinates, be expressed as $f=xy+g(z,w)$ for some $g\in (z,w)^2\subset  k[[z,w]]$ (Morse Lemma \cite[Section 11.1]{morse}).
Conversely, any isolated hypersurface given by the equation $f=xy+g(z,w)$ is a $cA_n$ singularity, where $n=\ord(g)-1$ and $\ord(g)$ is the degree of the lowest term of the power series of $g\in (z,w)^2$ \cite[Proposition 6.1 (e)]{burban-iyama-keller-reiten}.
Of particular interest are
 nodal singularities
(also called ordinary double
points)
given complete
locally by $xy + zw = 0$
and more generally ADE singularities,
see table \eqref{tab:ADE}.
Since $cA_n$ singularities are given by one equation
they are automatically Gorenstein.
 
Let $A$ be a complete local ring isomorphic to $k[[x,y,z,w]]/(f)$, where $f\in A$ is of the form $xy + g(z,w)$, for some $g\in k[[z,w]]$. One sees that $A$ has an isolated singularity at the origin if and only if the ring $k[[z,w]]/(g)$, which we denote by $A'$, has an isolated singularity at the origin. The latter condition is equivalent to $g$ being a nonconstant power series with no multiple factors.

Let $\br_{0}(A')$ be the number of irreducible components of $A'$. Here $0$ stands for the closed point $0\in\Spec(A')$. 

\begin{lemma}
We have a chain of equivalences
\begin{equation}\label{Ksg-of-cAn}
  \Z^{\br_{0}(A')-1} \simeq \Ksg_0(A') \simeq \Ksg_0(A) \simeq \Cl(A).
\end{equation}
\end{lemma}
\begin{proof}
We apply Proposition $\ref{prop:filtration-Ksg1}$ 
to the ELF schemes $\Spec(A')$ and $\Spec(A)$.
The first isomorphism is 2) of the Proposition,
the second one is 4) (or Theorem \ref{Thm:Knorrer} directly)
and the last isomorphism is 3).
\end{proof}

We call $\br_{0}(A):=\br_{0}(A')$, that is the number of irreducible components of $A'$, also \emph{branch number}
of $A$ (resp. $A'$).
More generally, if $X$ is a 
normal threefold with isolated $cA_n$ singularities, we denote by $\br_p(X)$ the branch number of $\wh{\OO}_{X,p}$ and we call it the \emph{branch number of $X$ at $p$}. The (total) \emph{branch number of $X$}, denoted $\br(X)$, is the sum of the $\br_p(X)$ running over $p\in\Sing(X)$.

It is well-known and easy to see that isolated $cA_1$ singularities are precisely $A_n$ threefold singularities (Morse Lemma \cite[Section 11.1]{morse}). 
More generally, the following table lists the local class groups of 
the ADE threefold singularities.

\begin{equation}\label{tab:ADE}
\begin{tabular}{| c | c | c | c |}
\hline
 Type & Equation & $\Cl(A)$ & $\br_0(A)$ \\ 
\hline
 $A_{2k} \; (k \ge 1)$ & $x^2 + y^2 + z^2+w^{2k+1}$ & $0$ & $1$ \\ 
\hline
 $A_{2k-1} \; (k \ge 1)$ & $x^2 + y^2 + z^2+w^{2k}$ & $\Z$ & $2$ \\ 
\hline
 $D_{2k} \; (k \ge 2)$ & $x^2 + y^2 + z^2w+w^{2k-1}$ & $\Z^2$ & $3$\\  
\hline
 $D_{2k-1} \; (k \ge 3)$ & $x^2 + y^2 + z^2w+w^{2k-2}$ & $\Z$  & $2$\\  
\hline
 $E_6$ & $x^2 + y^2 + z^3+w^4$ & $0$ & $1$ \\
\hline
 $E_7$ & $x^2 + y^2 + z^3+zw^3$ & $\Z$ & $2$\\
\hline
 $E_8$ & $x^2 + y^2 + z^3+w^5$ & $0$ & $1$\\
\hline
\end{tabular}
\end{equation}

\medskip

So far, we studied the complete local
geometry of $cA_n$ singularities;
global geometry of projective threefolds
with $cA_n$ singularities
in relation to their class groups, the so-called
defect $\delta$ and $\Kt_{-1}$,
is considered at the end of the next section.

\section{Class groups and $\Kt_{-1}$}
\label{sec:K-theory}

We recall that our schemes satisfy Orlov's ELF condition.
Furthermore, the words curve, surface, threefold are reserved for reduced
quasi-projective
schemes of pure dimensions one, two and three respectively.
Our goal in this section is to study $\Kt_{-1}$ for curves, surfaces
and threefolds. The results for threefolds with $cA_n$ singularities
are new, whereas results for curves and surfaces mostly go back to Weibel
\cite{weibel-surfaces}.

For a 
curve $C$ we denote by $\br_p(C)$ the branch number of $\wh{\OO}_{C,p}$ and call it \emph{branch number of $C$ at $p$} and by $\br(C)=\sum\br_p(C)$ the (total) \emph{branch number of $C$}. 
For example, the branch number of a node is two
and branch number of a cusp is one.
Let us now consider $\Kt_{-1}$ of a curve.

\begin{proposition}\label{prop:K_-1-curve}
Let $C$ be a connected 
curve. Then $\Kt_{-1}(C)$ is a free abelian group of rank 
\begin{equation}\label{eqn:K_-1-rank-curve}
 \br(C) -|\Sing(C)| -N +1,   
\end{equation}
where $N$ is the number of irreducible components of $C$. In particular, if $C$ has at most nodal singularities, then $\Kt_{-1}(C)$ is free abelian of rank $|\Sing(C)|-N+1$.
\end{proposition}

\begin{example}
Let $C \subset \P^2$ be a union of $N$ projective
lines intersecting in one point.
Then $\br(C) = N$, $|\Sing(C)| = 1$,
hence $\Kt_{-1}(C) = 0$.
\end{example}

\begin{proof} 
The statement will follow from a result of Weibel \cite[Lemma 2.3 (2)]{weibel-surfaces} by comparing the number of loops of the graph constructed in \cite{weibel-surfaces} with the number given in the statement. We will, however, give a different proof here using Grothendieck groups of the singularity category.

By Proposition \ref{prop:filtration-Ksg1} 1) and 2) we see that $\KKsg_0(C)\simeq
\bigoplus_{p} \Ksg_0(\wh{\OO}_{C, p})$ with each component being a free abelian group of rank $\br_p(C)-1$ and that $\Ksg_0(C)$ is a free abelian group of rank $N-1$. 
Using the short exact sequence $( \ref{SES})$ it is clear that the rank of $\Kt_{-1}(C)$ is just $\sum (\br_p(C)-1) -N-1$, which is equal to $(\ref{eqn:K_-1-rank-curve})$.
Furthermore, if $C$ is irreducible, the above argument also shows that
$\Kt_{-1}(C)$ is torsion-free.

To show that $\Kt_{-1}(C)$ is torsion-free in general, we proceed by induction on the number of irreducible components $N$ of $C$. 
Assume that $N\geq 2$ and let $C_0\subset C$ be an irreducible component of $C$. 
We may choose $C_0$ in such a way that $C - C_0$
is still connected (for that we can take $C_0$
to be a component with minimal number of intersections
with other components).
Consider the commutative diagram
\begin{equation}\label{diagram}
\xymatrix{
& 0 \ar[d] & 0 \ar[d] & 0 \ar[d] & \\
0\ar@{->}[r] & \Ksg_0(C\text{ on } C_0)\ar@{->}[r] \ar[d] & \KKsg_0(C\text{ on }C_0) \ar@{->}[r] \ar[d] & \Kt_{-1}(C\text{ on }C_0) \ar@{->}[r] \ar@{->}[d]  & 0 \\
0\ar@{->}[r] & \Ksg_0(C)\ar@{->}[r] \ar[d] & \KKsg_0(C) \ar@{->}[r] \ar[d] & \Kt_{-1}(C) \ar@{->}[r] \ar@{->}[d]  & 0 \\
0\ar@{->}[r] & \Ksg_0( C - C_0 ) \ar@{->}[r] \ar[d] & \KKsg_0(C - C_0) \ar@{->}[r] \ar[d] &  \Kt_{-1}(C - C_0)\ar@{->}[r] \ar[d] & 0 .\\
& 0  & 0 & 0  & \\
}
\end{equation}
The rows are exact by $\eqref{SES}$.
Let us now consider exactness of the columns.
By Proposition \ref{prop:filtration-Ksg1} 2) $\Ksg_0(C)\simeq \Z^{N-1}$, $\Ksg_0(C-C_0)\simeq\Z^{N-2}$ and $\Ksg_0(C\text{ on }C_0)\simeq\Z$ generated by the structure sheaves of the components of $C$, $C-C_0$ and $C_0$ respectively,
and the maps between the groups are the obvious ones,
so that the left column is split exact.
By Proposition $\ref{prop:filtration-Ksg1}$ 1) the middle column
is also split exact.

Applying the Snake Lemma to the first two columns
we get exactness of the right column.
Finally, by the induction hypothesis $\Kt_{-1}(C - C_0)$ is torsion-free and since $\Kt_{-1}(C\text{ on }C_0)$ is also torsion-free
because the top row is split exact, 
we obtain that $\Kt_{-1}(C)$ is torsion-free.

The statement about nodal curves
follows since the branch number of every node
is two, so that $\br(C) = 2|\Sing(C)|$.
\end{proof}

Recall that the \emph{dual graph} $\Gamma$
of a nodal curve $C$ is defined to be
the following (undirected) graph.
Vertices of $\Gamma$ correspond to
irreducible components of $C$.
Edges between distinct vertices
correspond to intersections of components. Finally, for every self-intersection point on a component
the corresponding vertex has a loop.
Usually $\Gamma$ is decorated
by indicating the genus of each
component, but we do not need this
for our purposes.

The following corollary implies
for example that $\Kt_{-1}$ of
a nodal cubic ($\Gamma$ has one
vertex with a loop),
or any cycle of smooth curves
($\Gamma$ is a cycle) is $\Z$
while $\Kt_{-1}$ of any tree
(that is $\Gamma$ is a tree)
of smooth curves is zero.

\begin{corollary}\label{cor:nodal}
Let $C$ be a curve with at most nodal singularities and let $\Gamma=\Gamma(C)$ be the dual graph of $C$, then
\[
\Kt_{-1}(C)\simeq\Z^{\lambda},
\]
where $\lambda =\lambda(\Gamma )$ is the 
first Betti number
of $\Gamma$.
\end{corollary}

\begin{proof}
Both sides are additive for finite
disjoint unions so we may assume that
$C$ is connected.

By definition of $\Gamma$, $N$ is its number of vertices and $|\Sing(C)|=\sum (\br_p(C)-1)$ is its number of edges. The result follows by Proposition $\ref{prop:K_-1-curve}$
since
\[
1 - \lambda = N - |\Sing(C)|
\]
so that $\lambda = |\Sing(C)| - N+ 1$.
\end{proof}

Let us consider the higher dimensional case.

\begin{definition}\label{def:mnf}
Let $X$ be 
normal irreducible
with at most isolated singularities. We say that $X$ 
has enough Weil divisors
if the natural map $\Cl(X)\to\bigoplus_p\Cl(\wh{\OO}_{X,p})$ is surjective, where the direct sum runs over all $p\in\Sing(X)$.

If $X$ is a threefold with isolated nodal singularities,
and $X$ has enough Weil divisors, we say that $X$
\emph{is maximally nonfactorial}.
\end{definition}

\begin{proposition}\label{prop:max-nonfact}
Let $X$ be normal 
irreducible with at most isolated singularities. Assume
that $\Kt_{-1}(X) = 0$. Then $X$ 
has enough Weil divisors.
\end{proposition}

\begin{proof}
Let $X' = \Spec\left(\prod_{p\in\Sing(X)} \wh{\OO}_{X,p}\right)$.
By Proposition \ref{prop:filtration-Ksg1} 1), $\KKsg_0(X) \simeq \Ksg_0(X')$.

Since $X$ is irreducible, we have $\Gr^0 \Ksg_0(X) = 0$, so that
$\Ksg_0(X) = F^1 \Ksg_0(X)$.
Similarly, since $X'$ is a disjoint union of irreducible components,
$\Gr^0 \Ksg_0(X') = 0$
and $\Ksg_0(X) = F^1 \Ksg_0(X')$.
In particular as $\Ksg_0(X) \to \KKsg_0(X)$ is surjective
by \eqref{SES}, we also get that $\Gr^1 \Ksg_0(X) \to \Gr^1 \Ksg_0(X')$
is surjective.

By Proposition \ref{prop:filtration-Ksg1} 3) we have a commutative diagram
\begin{equation}\label{eq:Ksg0-diagram}
\xymatrix{
\Gr^1 \Ksg_0(X) \ar[r] \ar[d]^\simeq & \Gr^1 \Ksg_0(X') \ar[d]^\simeq \\
\Cl(X) / \Pic(X) \ar[r] & \Cl(X') \\
}
\end{equation}
where we used that $\Pic(X') = 0$.
Since we know that the top horizontal arrow is surjective,
the bottom horizontal arrow is surjective as well,
which 
means that $X$ has enough Weil divisors.
\end{proof}

If the local singularity Grothendieck groups are generated by codimension
one cycles, then $\Kt_{-1}(X)$ is controlled by
codimension one cycles as well:

\begin{proposition}\label{prop:Class-groups-singularity}
Let $X$ be 
normal irreducible with at most isolated singularities and such that $F^2\Ksg_0(\wh{\OO}_{X,p})=0$ for all $p$ in $\Sing(X)$. Then there is an isomorphism $\Ksg_0(\wh{\OO}_{X,p})\simeq\Cl(\wh{\OO}_{X,p})$ for all $p$ in $\Sing(X)$ and  an exact sequence
\begin{equation}\label{eqn:main-exact-sequence}
  0 \to \Pic(X) \to \Cl(X) \to \bigoplus_{p\in\Sing(X)}\Cl(\wh{\OO}_{X,p}) \to \Kt_{-1}(X) \to 0.  
\end{equation}
In particular, 
$\Kt_{-1}(X)=0$
if and only if $X$ has enough Weil divisors.
\end{proposition}

\begin{proof}
We keep the notation of the previous proof.
By \eqref{eq:Fi-relative-completion} we have injective maps
\[
F^i\Ksg_0(X) \to F^i\Ksg_0(X').
\]
Thus using $F^2\Ksg_0(\wh{\OO}_{X,p})=0$ for all $p$ in $\Sing(X)$, so that $F^2\Ksg_0(X') = 0$ we see that $F^2\Ksg_0(X)=0$.
Therefore diagram \eqref{eq:Ksg0-diagram}
is isomorphic to
\begin{equation}\label{eq:Ksg0-diagram2}
\xymatrix{
\Ksg_0(X) \ar[r] \ar[d]^\simeq & \Ksg_0(X') \ar[d]^\simeq \\
\Cl(X) / \Pic(X) \ar[r] & \Cl(X') \\
}
\end{equation}

The upper horizontal map is injective with cokernel $\Kt_{-1}(X)$ by \eqref{SES},
hence we get a short exact sequence
\[
0 \to \Cl(X)/\Pic(X) \to \Cl(X') \to \Kt_{-1}(X) \to 0,
\]
which implies \eqref{eqn:main-exact-sequence}.
\end{proof}

If $S$ is a normal surface, then the exact sequence $(\ref{eqn:main-exact-sequence})$ applies to $S$ by Proposition $\ref{prop:filtration-Ksg1}$ 1), and recovers a result of Weibel \cite[Corollary 5.4]{weibel-surfaces}.
Furthermore we have the following result.

\begin{proposition}\label{prop:surface-Br}
If $X$ is normal rational projective surface with rational singularities,
then we have an isomorphism $\Kt_{-1}(X) \simeq \Br(X)$.
\end{proposition}
\begin{proof}
The proof is a combination
of a result of Weibel computing $\Kt_{-1}(X)$
with a result of Bright computing $\Br(X)$.

Let $\pi:\wt{X}\to X$ be a resolution of singularities of $X$, such that the exceptional divisor $E=\pi^{-1}(\Sing(X))$ is a normal crossing divisor.
By Artin \cite{artin}, $E$ is a tree of 
smooth rational curves. 
Let $N$ be the
number of irreducible components of $E$.

By
\cite[Example 2.13 and Proposition 5.1]{weibel-surfaces} there is an exact sequence
\begin{equation}\label{eqn:secret-exact-sequence}
0\to\Pic(X)\to\Pic(\wt{X})\to\Pic(E)\to\Kt_{-1}(X)\to 0 .
\end{equation}
It is well known that $\Pic(E)\simeq \Z^{N}$ spanned by the tautological bundles of the components of $E$. This group is also canonically isomorphic to the
free abelian
group $\mathbf{E}^{\star}$ 
generated by the components of the exceptional divisor
defined in \cite{bright}, and comparing 
\eqref{eqn:secret-exact-sequence} to 
\cite[Proposition 1]{bright} (the setup in \cite{bright}
includes minimality of the resolution, but it is not required in the proof), where
we use that $\Br(\wt{X}) = 0$
since $\wt{X}$ is a smooth projective
rational surface yields $\Kt_{-1}(X)\simeq\Br(X)$.
\end{proof}

The following result allows to compute $\Kt_{-1}$ of threefolds
with isolated compound $A_n$ singularities (in particular for nodal
threefolds), in terms
of their Picard group, Class group and the branch number defined 
in subsection \ref{subsec:cAn}.

\begin{corollary}\label{cor:K_{-1}}
Let $X$ be 
normal threefold with at most isolated $cA_n$ singularities. Then we have an exact sequence
\[
0 \to \Pic(X) \to \Cl(X) \to \Z^L \to \Kt_{-1}(X) \to 0,
\]
where $L=\br(X)-|\Sing(X)|$ is the difference between the branch number and the number of the singular points of $X$. In particular, if $X$ has at most nodal singularities, then $L= |\Sing(X)|$. 
\end{corollary}

\begin{proof}
Using $(\ref{Ksg-of-cAn})$ we obtain that $\Cl(\wh{\OO}_{X,p})\simeq\Z^{\br_p(X)-1}$. 
Since $L=\sum_{p \in \Sing(X)} (\br_p(X)-1)$, the result follows from Proposition $\ref{prop:Class-groups-singularity}$.
\end{proof}

\begin{definition}\label{def:defect}
Let $X$ be a 
normal threefold with at most isolated $cA_n$ singularities. We define the defect $\delta$ of $X$ by
\[
\delta:=\rk\,{\Cl(X)/ \Pic(X).}
\]
\end{definition}

\begin{remark}
Note that the defect is well-defined by Corollary $\ref{cor:K_{-1}}$. It was first defined by Clemens in \cite{clemens} for double solids, then by Werner \cite{werner} for nodal 3-dimensional hypersurfaces and later it was extended by Rams to 3-dimensional hypersurfaces with ADE singularities \cite{rams}. By \cite[Corollary 2.32]{clemens} and \cite[Theorem 4.1]{rams} one sees that the classical definition of the defect agrees with Definition $\ref{def:defect}$.
\end{remark}

\begin{remark}
Let $X$ be as in Definition $\ref{def:defect}$.
We can rewrite
 Corollary $\ref{cor:K_{-1}}$
 as a short exact sequence
\begin{equation}\label{eq:delta-L}
0 \to \Z^\delta \to \Z^L \to \Kt_{-1}(X) \to 0
\end{equation}
where $L=\br(X)-|\Sing(X)|$.
Explicitly, the first group $\Z^\delta \simeq \Cl(X) / \Pic(X)$
is generated by the classes of Weil divisors which are not Cartier,
the second group $\Z^L$ is the sum of local class groups of the
singular points, and the map between them is given by restricting 
Weil divisors to the local class groups.

By definition, $X$ is factorial if and only if $\delta =0$. On the other hand, if $X$ has enough Weil divisors, 
then $\delta=L$. 

It is worth noticing that if $L=0$, then $X$ is factorial and 
has enough Weil divisors
at the same time. Indeed, from \eqref{eq:delta-L} we see that $\delta =0$, as well as $\Kt_{-1}(X)=0$. 
For isolated $cA_n$ singularities
$L = 0$ if and only if all branch numbers of the singular
points are equal to one, that is singularities are
of type $xy+g(z,w)=0$, where $g(z,w)$ is irreducible.
For example, this is the case for
$A_{2k}$ singularities, see \eqref{tab:ADE}.
\end{remark}

We collect examples for known defect
of nodal threefolds.
Recall that we call
nodal threefolds with enough Weil
divisors, maximally nonfactorial.

\begin{example}
If $X$ is a nodal quadric threefold in $\P^4$,
then $\Pic(X) = \Z$ generated by the class
of the hyperplane section $H$; $\Cl(X) = \Z^2$, 
generated by the two planes $D_1$, $D_2$ passing
through the singular point, so that $H = D_1 + D_2$.
Therefore, $\delta = 1$, $L = 1$, the first
map in \eqref{eq:delta-L} is an isomorphism,
$\Kt_{-1}(X) = 0$ and $X$ is maximally nonfactorial.
\end{example}

\begin{example}\label{example:hypersurface}
Let $X$ be a nodal hypersurface in $\P^4$ or a nodal double cover of $\P^3$, which is not the nodal quadric hypersurface in $\P^4$. Let $r$ be the number of nodes of $X$. The defect $\delta$ in these cases has been studied in detail and it is known that $\delta < r$ (see \cite[Definition 1 and Theorem 9]{cynk} for the hypersurface case and \cite[Corollary 2.32]{clemens} for double solids). Thus by Corollary $\ref{cor:K_{-1}}$ we get that $\rk\, \Kt_{-1}(X)=r-\delta >0$, i.e. 3-dimensional nodal hypersurfaces and nodal double covers of $\P^3$ are never maximally nonfactorial, except for the 3 dimensional nodal quadric hypersurface.
\end{example}

\begin{lemma}\label{lem:rank-K_-1-blowup}
Let $\pi: \wt{X} \to X$ be a small resolution
of a nodal projective threefold with $r$ nodes.
Assume that $\wt{X}$ is obtained 
as a blow up of a smooth projective threefold
$Y$ in $\mu$
points.
Let $\rho_X$, $\rho_Y$ are Picard ranks of $X$ and $Y$
respectively.
Then 
we have
\[
\rk \, \Kt_{-1}(X) = r - \delta = r - \mu + \rho_X - \rho_Y ,
\]
where $\delta = \mu +\rho_Y - \rho_X$ is the defect of $X$.
\end{lemma}
\begin{proof}
Since $\pi:\wt{X}\to X$ is a small resolution, we have that $\Cl(\wt{X})\simeq\Cl(X)$. Moreover, since $\wt{X}$ is a smooth blow up of $Y$ at $\mu $ points, we see further that $\Cl(\wt{X})\simeq \Cl(Y)\oplus\Z^{\mu}$. The result is then a direct consequence of Corollary $\ref{cor:K_{-1}}$.
\end{proof}

\begin{example}\label{example:del-Pezzo-3folds}
According to Prokhorov \cite[Theorem 7.1]{prokhorov}, del Pezzo threefolds
of degree $1 \le d \le 5$, 
that is Fano threefolds of Picard rank one and
index two, with maximal
class group rank are obtained by blowing up
$8 - d$ general points $P_i$ on $\P^3$ followed by blowing down proper preimages of lines and twisted cubics passing through the points $P_i$
(the latter contraction is realized as an algebraic variety
by taking half-anticanonical model of the blow up).
The number of nodes of $X$ is 28 for $d=1$, 16 for $d=2$ and $8-d \choose 2$ for $3\leq d\leq 5$ \cite[Theorem 7.1 (iii)]{prokhorov}.

Using Lemma $\ref{lem:rank-K_-1-blowup}$ we see that the rank of $\Kt_{-1}$ is 21 for $d=1$, 10 for $d=2$ and $\frac{(8-d)(5-d)}{2}$ for $3\leq d\leq 5$. In particular, the cases $1\leq d \leq 4$ are not maximally nonfactorial (see also table $(\ref{table})$).
\end{example}

\section{Kawamata type semiorthogonal decompositions}
\label{sec:Kawamata}

In this section $X$ is a Gorenstein projective variety. 
The following definition is motivated by \cite{kawamata1} and \cite{kawamata2}.

\begin{definition}\label{def:KSOD}
We say that $X$ has a Kawamata type semiorthogonal decomposition 
if
\[
\Db(X) = \langle \AA , \BB_1, \ldots ,\BB_m\rangle
\]
is an admissible semiorthogonal decomposition, such that $\AA\subset \Dperf(X)$ and the $\BB_j$'s are equivalent to $\Db(R_j)$,
where the $R_j$'s are (possibly noncommutative) finite-dimensional $k$-algebras.
\end{definition}

\begin{remark}
Note that a smooth projective variety $X$ trivially admits a 
Kawamata type decomposition. Indeed, in this case $\Dperf(X) = \Db(X)$
and we can set $m = 0$, $\AA = \Db(X)$.
\end{remark}

\begin{remark}\label{rmk:rearrangement}
Any admissible decomposition of $\Db(X)$ into components which are subcategories of $\Dperf(X)$
and components equivalent to $\Db(R)$ can be rearranged to make a Kawamata decomposition. This follows from a result of Bondal and Kapranov \cite[Lemma 1.9]{bondal-kapranov}, which implies more generally that any admissible semiorthogonal decomposition can be mutated.
\end{remark}

\begin{theorem}\label{thm:kawamata-sod}
If $X$ admits a Kawamata type semiorthogonal decomposition 
\[
\Db(X) = \langle \AA ,\BB_1 \ldots ,\BB_m \rangle,
\]
then the following holds.

\noindent 1) There is an admissible semiorthogonal decomposition 
\[
\Dperf(X) = \langle \AA, \BB_1\cap\Dperf(X),\ldots , \BB_m \cap \Dperf(X) \rangle,
\]
and $\BB_j \cap \Dperf(X)$ is equivalent to $\Dperf(R_j)$ for all $1\leq j\leq m$.
The Serre functor on $\Dperf(X)$ induces Serre functors on $\AA$ and on all $\BB_j\cap\Dperf(X)$. 

\noindent 2) The finite-dimensional $k$-algebras $R_j$ are Gorenstein.

\noindent 3) There is an equivalence of singularity categories
$
\Dsg(X) \simeq \langle \Dsg(R_{1}),\ldots ,\Dsg(R_{m})\rangle 
$. Furthermore, if $X$ has only isolated singularities, then the decomposition above is completely orthogonal, that is
\begin{equation}\label{eqn:Dsg-decomp}
 \Dsg(X) \simeq \Dsg(R_{1}) \oplus \ldots \oplus \Dsg(R_{m}) \simeq \Dsg(R_1\times\ldots\times R_m).
\end{equation}
\end{theorem}

\begin{proof}
\noindent 1) The decomposition and its admissibility follows immediately from 
Orlov's characterization of perfect complexes as homologically finite objects
\cite[Proposition 1.10 and 1.11]{orlov-hf}. 
Moreover, by the analogous characterization
of $\Dperf(R_j)$ in $\Db(R_j)$ 
\cite[Proposition 2.18]{iyama-wemyss} and by admissibility of $\BB_j$ it is easy to see that $\Dperf(R_j)=\BB_j\cap\Dperf(X)$.
By Lemma \ref{lem:Serre-adm} it follows that the components $\AA$ and $\BB_j\cap\Dperf(X)$ have Serre functors.

\noindent 2) By 1) we see that $\Dperf(R_j)$ has a Serre functor, and by Lemma
\ref{lem:Gor-Serre} this is equivalent to $R_j$ being Gorenstein.

\noindent 3) The decomposition of $\Dsg(X)$ follows by \cite[Proposition 1.10]{orlov-hf}.
Let us assume now that $X$ has isolated singularities.
By Proposition \ref{prop:CY}, $\Dsg(X)$ is a Calabi-Yau category.
By a standard argument going back to Bridgeland \cite{bridgeland} it is easy to see that in this case decomposition is completely orthogonal. The second equivalence in $(\ref{eqn:Dsg-decomp})$ is clear.
\end{proof}

\begin{corollary}\label{cor:K_{-1}-obstruction}
If $X$ admits a Kawamata type semiorthogonal decomposition, then $\Dsg(X)$ is idempotent complete, or, equivalently, $\Kt_{-1}(X) = 0$.
\end{corollary}
\begin{proof}
By Theorem \ref{thm:kawamata-sod} (3),
the singularity category $\Dsg(X)$
admits a semiorthogonal decomposition
into singularity categories of finite-dimensional
$k$-algebras.
Each of these categories is idempotent complete
by Lemma \ref{lem:algebras-complete},
and
Lemma \ref{lem:semiorth-complete}
implies that $\Dsg(X)$ is itself idempotent
complete.
The vanishing of $\Kt_{-1}(X)$ follows
from Lemma $\ref{lem:SES_K_{-1}}$.
\end{proof}

\begin{example}
If $X$ has trivial canonical bundle,
then it admits a Kawamata decomposition if and only if $X$ is smooth.
Indeed if we assume that $X$ has a Kawamata decomposition, it follows by Theorem $\ref{thm:kawamata-sod}$ 1) that there is an induced semiorthogonal decomposition of $\Dperf(X)$. While the Serre functor on $\Dperf(X)$ is just the shift by $n=\dim(X)$, we see that for a finite-dimensional algebra as in Definition $\ref{def:KSOD}$, we have that 
\[
\Hom_{R_j}(R_j, R_j)={\Ext^n_{R_j}}(R_j, R_j)^{\star},
\]
which is only possible if $\dim(X)=0$
or $m=0$. In both cases $X$ is smooth.
\end{example}

The next proposition shows that admissibility is automatic in the case when we have only one algebra.

\begin{proposition}\label{prop:admiss}
Assume that $X$ has a semiorthogonal decomposition 
\[
\Db(X)=\langle \AA, \BB\rangle ,
\]
where $\AA \subset \Dperf(X)$, $\BB \simeq \Db(R)$ 
(resp. $\BB \subset \Dperf(X)$, $\AA \simeq \Db(R)$), and $R$ is a finite-dimensional
algebra.
Then $\AA$ and $\BB$ are admissible subcategories in $\Db(X)$
so that the semiorthogonal decomposition $\Db(X) = \langle \AA, \BB \rangle$
(resp. $\Db(X) = \langle \BB \otimes \omega_X, \AA \rangle$) is of Kawamata type.
\end{proposition}

\begin{proof}
This follows from Lemma \ref{lem:Gor-admiss}.
\end{proof}

\begin{remark} Kawamata type decompositions generalize
tilting objects in the following sense.
Recall that a classical tilting object $\EE$ of $\DD(\Qcoh(X))$ is a perfect complex of $\DD(\Qcoh(X))$, such that it generates $\DD(\Qcoh(X))$ (i.e. if $\Hom(\EE ,\FF)=0$, then $\FF\simeq 0$) and such that $\Hom(\EE ,\EE[i])=0$ for all $i\neq 0$. It is well known that, if $\DD(\Qcoh(X))$ possesses a classical tilting object $\EE$, then there is an equivalence $\DD(\Qcoh(X))\simeq\DD(\Mod\textrm{-}R)$ which restricts to an equivalence $\Db(X)\simeq\Db(R)$, where $R$ is the finite-dimensional algebra $\End(\EE)$ (see e.g. \cite[Theorem 7.6 (2)]{hille-van-den-bergh}). This means that $X$ has a Kawamata semiorthogonal decomposition with trivial $\AA\subset\Dperf(X)$ part as soon as $\DD(\Qcoh(X))$ has a classical tilting object.
\end{remark}

We collect the known examples of Gorenstein
projective
varieties with Kawamata type semiorthogonal decompositions.
We start in dimension one.

\begin{theorem}[Burban \cite{burban}]\label{thm:burban}
Let $X$ be a nodal tree of projective lines, 
that is a connected nodal curve with all irreducible components
isomorphic to $\P^1$ and with the dual graph $\Gamma$ of
$X$ forming a tree.
Then $\Db(X)$ has a tilting object, and furthermore admits
a Kawamata type
semiorthogonal decomposition
\[
\Db(X) = \langle \OO_X, \Db(R_\Gamma) \rangle.
\]
The algebra $R_\Gamma$ is
the path algebra of the quiver $Q$
with relations, obtained
by the following construction from $\Gamma$: $Q$ has the same vertices as $\Gamma$
and for
each two vertices $p$, $q$ in $\Gamma$
connected by an edge there
is an arrow $a$ from $p$ to $q$
and an arrow $a^*$ from $q$ to $p$ in $Q$. The relations
are that all compositions $a a^*$
and $a^* a$ are set equal zero.
\end{theorem}

\begin{proof}
The proof is the same as that of Theorem 2.1 in \cite{burban},
where only chains of projective lines are considered.
\end{proof}

\begin{example}\label{example:burban-example}
Let $X=X_1\cup X_2$ be the $A_2$ tree of projective
lines, that is a union of 2 copies of $\P^1$ intersecting transversely. 
Then the algebra $R_\Gamma$ in Theorem \ref{thm:burban} has the form
\begin{equation}\label{eqn:burban-kawamata-algebra}
\begin{tikzcd}
 & R: & 1  \arrow[r,bend right,swap,"a"] & 2  \arrow[l,bend right,swap,"a^*"] & a^* a=0,\;\;\;\; a a^*=0 .
\end{tikzcd}
\end{equation}

By Theorem \ref{thm:kawamata-sod} 
we have $\Dsg(X) \simeq \Dsg(R) \simeq 
\ul{\MCM}(R)$
(we used Buchweitz' equivalence
between the singularity category
and the stable category of MCM modules
in the Gorenstein case for the second
equivalence).
On the other hand, by
\cite[Proposition 1.14]{orlov-sing-1},
we have an equivalence
\[
\Dsg(X) \simeq \Dsg(A)  
\]
where $A = k[x,y]/(xy)$ is the ring
considered in Example \ref{ex:CY0}.

Explicitly the generators
of the singularity category considered
in Example \ref{ex:CY0} correspond to 
the two MCM $R$-modules which are the
two simple modules given by
the vertices of the quiver.
\end{example}

The following result
pins down the interplay between
geometry, Kawamata decompositions, and $\Kt_{-1}$
in the case of nodal curves.

\begin{corollary}\label{cor:nodal-curve-characterisation}
Let $C$ be a connected nodal projective curve
such that all its irreducible components
are rational curves. Then the following are equivalent:

\noindent 1) $C$ is a nodal tree of projective lines.

\noindent 2) $\Db(C)$ admits a Kawamata type semiorthogonal decomposition.

\noindent 3) $\Kt_{-1}(C)=0$.
\end{corollary}

\begin{proof}
1) $\implies$ 2)
is Theorem $\ref{thm:burban}$.

2) $\implies$ 3) follows from 
Corollary $\ref{cor:K_{-1}-obstruction}$.

3) $\implies$ 1) follows from Corollary $\ref{cor:nodal}$.
\end{proof}

The following result gives a source of examples of Kawamata type semiorthogonal decompositions in dimension two.

\begin{theorem}[Karmazyn-Kuznetsov-Shinder \cite{karmazyn-kuznetsov-shinder}]\label{thm:toric-surfaces}
Let $X$ be a projective Gorenstein toric surface. Let $n_1 ,\ldots , n_m$ be the orders of the cyclic quotient singularities of $X$. Then $X$ has a Kawamata type semiorthogonal decomposition if and only if $\Kt_{-1}(X)=0$ and in this case the decomposition is of the form 
\[
\Db(X)\simeq\langle\AA ,\Db(R_1),\ldots\Db(R_m)\rangle ,
\]
where the category $\AA\subset\Dperf(X)$ is a collection of exceptional objects and such that $R_i=k[z]/(z^{n_i})$.
\end{theorem}

\begin{proof}
By Proposition \ref{prop:surface-Br} 
we have $\Br(X) = \Kt_{-1}(X)$.
If $\Kt_{-1}(X) = 0$, the semiorthogonal decomposition in Theorem $\ref{thm:toric-surfaces}$ is \cite[Corollary 5.10]{karmazyn-kuznetsov-shinder} and admissibility of the components is \cite[Theorem 2.12]{karmazyn-kuznetsov-shinder}.

Conversely, existence of a Kawamata
type decomposition implies $\Kt_{-1}(X) = 0$
by Corollary \ref{cor:K_{-1}-obstruction}.
\end{proof}

For threefolds, we have the following two Fano examples due to Kawamata.

\begin{example}[Kawamata]\label{example:Kawamata}

\noindent (1) Let $X$ be the nodal quadric threefold in $\P^4$ with the equation $xy-zw=0$. In \cite[Example 5.6]{kawamata1} (see also \cite[Example 7.1]{kawamata2}) it has been shown that there is an admissible semiorthogonal decomposition
\[
\Db(X)\simeq\langle\OO_X(-2H),\OO_X(-H),\Db(R),\OO_X\rangle ,
\]
where $\OO_X(H)$ is a hyperplane section bundle and $R$ is the same algebra as $(\ref{eqn:burban-kawamata-algebra})$ in Example $\ref{example:burban-example}$. By Remark
\ref{rmk:rearrangement}, $X$ has a Kawamata type decomposition. 
    
\noindent (2) Let $\wt{X}$ be the blow up of 
two points in $\P^3$ and let $L \subset \wt{X}$ be the strict transform of a line passing through the 
two points. Let $X$ be the contraction of $L$ to a node
given by the half-anticanonical embedding in $\P^7$. 
By \cite[Example 7.2]{kawamata2} $X$ has a Kawamata type decomposition
\[
\Db(X)\simeq\langle\AA ,\Db(R)\rangle\simeq\langle\OO_X(C_1),\ldots ,\OO_X(C_5),\Db(R)\rangle,
\]
where the $\OO_X(C_i)$'s are line bundles on $X$ which are push-forwards of line bundles from $\wt{X}$, and $R$ is again the algebra $(\ref{eqn:burban-kawamata-algebra})$.
\end{example}

\begin{remark}
Let us contemplate here on the fact that the algebra occurring in Examples $\ref{example:Kawamata}$ (1) and (2) coincides with the algebra which shows up in the union $C$ of 2 rational curves intersecting at a node (Example $\ref{example:burban-example}$). This observation is related to Kn\"orrer periodicity. More concretely, the singularity category of $C$ will agree with the singularity category of Kawamata's examples via Kn\"orrer periodicity (Theorem $\ref{Thm:Knorrer}$). On the other hand, Kn\"orrer periodicity can be realized via a blow up construction (see \cite[Proof of Proposition 1.30]{pavic-shinder} and Remark \ref{rmk:blowup}) 
which provides an explicit link between Kawamata type decompositions of ordinary double points with the same parity. 

More generally, this viewpoint of blowing up projective Gorenstein schemes at locally complete intersection subschemes will provide further examples of Kawamata type semiorthogonal
decomposition, as shown in Section
\ref{sec:blowup}.
\end{remark}

In the next two examples
we consider typical singular threefolds:
hypersurfaces, double covers and contractions of blow ups.

\begin{example}\label{ex:hypers}
Nodal hypersurfaces in $\P^4$ of degree $d\geq 3$ and nodal double covers of $\P^3$ branched
in a surface of degree at least four have 
no Kawamata type decomposition by Example $\ref{example:hypersurface}$ and Corollary $\ref{cor:K_{-1}-obstruction}$.

This generalizes an example of Kawamata \cite[Example 7.8]{kawamata2},
constructed as follows. 
One considers a cubic threefold with two nodes $p, q \in X_0$;
it is well-known that such cubics are always factorial. 
Let $X$ be the blow up of $q$.
In the commutative square from \eqref{eqn:main-exact-sequence}
\[\xymatrix{
\Cl(X) \ar[d] \ar[r] & \Cl(\wh{\OO}_{X, p}) \ar[d] 
\\
\Cl(X_0)  \ar[r] & \Cl(\wh{\OO}_{X_0, p}) \oplus \Cl(\wh{\OO}_{X_0, q})\\ 
}\]
the bottom horizontal map is zero, the right vertical map is an embedding of a direct
summand,
hence the top horizontal map is also zero, so that
$X$ is factorial as well.
From \eqref{eqn:main-exact-sequence}
we deduce that $\Kt_{-1}(X)=\Z$, so that $X$
has no Kawamata type decomposition 
by Corollary \ref{cor:K_{-1}-obstruction}.
\end{example}

\begin{example}\label{example:del-Pezzo-3folds-KSOD}
Del Pezzo threefolds as in Example $\ref{example:del-Pezzo-3folds}$ of degree $1 \le d \le 4$ have no Kawamata type decomposition. This follows from Corollary $\ref{cor:K_{-1}-obstruction}$. 

This relates to a question of Kawamata about derived categories of blow ups of $\P^3$ in more than $2$ points \cite[Remark 7.5]{kawamata2}. We have shown that the half-anticanonical contraction of a blow up of $\P^3$ in $4$ or more points has no Kawamata type decomposition. The remaining case, that is the nodal del Pezzo threefold of rank $5$, seems to be the most interesting one, as we cannot detect obstructions with our methods. The following table gives a summary:
\begin{align}\label{table}
 \begin{tabular}{| c | c | c | c | c | c |}
\hline
 $d(X)$ & $|\Sing(X)|$ & $\rk\, \Pic(X)$ & $\rk\, \Cl(X)$ & $\rk\,\Kt_{-1}(X)$ & Kawamata decomp.\\ 
\hline
 1    &   28   &   1   &   8   &   21   &   No\\ 
\hline
 2    &   16   &   1   &   7   &   10   &   No\\ 
\hline
 3    &   10   &   1   &   6   &   5    &   No\\  
\hline
 4    &   6    &   1   &   5   &   2    &   No\\  
\hline
 5    &   3    &   1   &   4   &   0    &   ?\\
\hline
 6    &   1    &   2   &   3   &   0    &   Yes\\
\hline
\end{tabular}   
\end{align}

Here the $d=6$ case refers to Example $\ref{example:Kawamata}$ 2).
\end{example}

\section{Kawamata decompositions, $\Kt_{-1}$ and blow ups}\label{sec:blowup}

We start with the following well-known result.

\begin{theorem}[Thomason, Orlov]\label{thm:blowup}
Let $X$ be a Gorenstein projective variety
and $Z \subset X$ a locally complete
intersection closed subvariety
of pure codimension $c$.
Let $\pi : \wt X\to X$ be the blow up
of $X$ with center $Z$.
Then 
there is an admissible semiorthogonal
decomposition 
\[
\Db(\wt{X}) = \langle 
\underbrace{\Db(Z), \dots, \Db(Z)}_{c-1}, 
\Db(X) \rangle
\]
and for all $j \in \Z$ we have
\[\Kt_j(\wt{X}) \simeq \Kt_j(X) \oplus \Kt_j(Z)^{\oplus(c-1)}.\]
\end{theorem}

\begin{proof}
The semiorthogonal decomposition is proved in the same way
as \cite{orlov-monoidal} (see also \cite[Theorem 6.9]{bergh-schnurer} or
\cite[Corollay 3.4]{jiang-leung}).
Let us give just a few words on the well-definedness of the functors involved.


Indeed, note that all the morphisms $Z\subset X$, $E\subset\wt{X}$, where $E$ is the exceptional locus of $\pi$, $p: E\to Z$, and $\pi :\wt{X}\to X$ are proper
of finite Tor dimension. This is because the first two morphisms are regular embeddings and the morphism $p: E\to Z$ is a projective bundle. For the blow up $\pi:\wt{X}\to X$, we can write it locally as a composition $\wt{U}\to\P(\EE)\to U$ of a regular embedding and a projective bundle, where $\EE$ is a vector bundle on $U \subset X$ such that the zero locus of global section $0\neq s\in H^0(\EE^{\vee})$ coincides with $Z$. Thus the pushforward and pull-back functors of these morphisms on $\Db$ are well-defined.

Finally the decomposition for $\Kt$-theory is proved
by Thomason \cite[Theorem 2.1]{thomason-blowup},
and it also can be deduced from the semiorthogonal decomposition
lifted to dg-enhancements of the relevant categories,
and applying Schlichting's machinery \cite{schlichting2, schlichting}.
\end{proof}

\begin{remark}\label{rmk:blowup}
In this paper we mostly deal with isolated singularities,
however the
blow up of a smooth variety in a center with isolated singularties
does not necessarily have isolated singularities. For example, if we blow up $\A^3$ at the thick point given by the ideal $(x,y,z^2)$, then one can see that the singular locus of the blow up is 1-dimensional.

A local computation shows however that, if $Z\subset X$ is a locally complete intersection of codimension $2$ in a smooth variety $X$
and such that $Z$ has at most isolated hypersurface
singularities 
given complete
locally by an ideal
$\big(f(x_1,\ldots ,x_{n-1}), x_{n}\big)
\subset k[[x_1, \dots, x_{n}]]$, 
then the blow up $\wt{X}\to X$ along $Z$ has at most isolated hypersurface singularities given by the ideal
$\big(x_{n} \cdot x_{n+1}+
f(x_1,\ldots , x_{n-1})\big) 
\subset k[[x_1, \dots, x_{n+1}]]$.
\end{remark}

\begin{corollary}\label{cor:blowup-ksod}
Under the conditions of Theorem \ref{thm:blowup} if both $X$ and $Z$
admit Kawamata decompositions, so does
$\wt{X}$.
\end{corollary}
\begin{proof} While all the components are admissible in $\Db(\wt{X})$, we can use Remark $\ref{rmk:rearrangement}$ to rearrange them to obtain the form as in Definition $\ref{def:KSOD}$.
\end{proof}

\begin{corollary}\label{cor:nodal-blowup-characterisation}
Let $X$ be a smooth projective threefold and let $C\subset X$ be a disjoint union of nodal curves such that each irreducible components of $C$ is a rational curve. 
Then $\wt{X}$ admits a Kawamata type semiorthogonal decomposition if and only if $C$ is a disjoint union of nodal trees of smooth rational curves.
\end{corollary}

\begin{proof}
If $C$ is a disjoint union of nodal trees of smooth rational curves,
then the blow up has 
a Kawamata type decomposition
by Corollary $\ref{cor:blowup-ksod}$ and Theorem \ref{thm:burban}.

Conversely, if the blow up admits a Kawamata type decomposition,
then $\Kt_{-1}(\wt{X}) = 0$ by Corollary \ref{cor:K_{-1}-obstruction}
hence $\Kt_{-1}(C) = 0$ by
Theorem \ref{thm:blowup}
and finally $C$ is a nodal tree by 
Corollary \ref{cor:nodal-curve-characterisation}.
\end{proof}

\begin{example}\label{example:nodal-chain-blowup}
If $X$ is a smooth projective threefold,
and $C$ is a disjoint union of nodal trees
of projective lines, then
the blow up $\wt{X} = \Bl_C(X)$ is a threefold with ordinary double
points (see Remark $\ref{rmk:blowup}$) 
and by Corollary \ref{cor:nodal-blowup-characterisation}
it admits a Kawamata type semiorthogonal
decomposition. 

On the other hand, if $C$ is nodal and irreducible (of arbitrary genus), 
then the blow up
of $X$ in $C$ does not have a Kawamata type decomposition
by Corollary \ref{cor:K_{-1}-obstruction}
since $\Kt_{-1}(\wt{X}) = \Kt_{-1}(C) \ne 0$ where
we used Theorem \ref{thm:blowup} and 
Corollary \ref{cor:nodal}.
\end{example}

\providecommand{\arxiv}[1]{{\tt{arXiv:#1}}}

\end{document}